\documentclass[11pt]{article}
\usepackage{amsmath,amssymb,amsthm}
\usepackage{graphicx,subfigure,float,url}
\usepackage{pdfsync}
\usepackage[english]{babel}
\usepackage[utf8]{inputenc}
\usepackage{tipa}
\usepackage{tikz}

\newlength\imagewidth
\newlength\imagescale

\usepackage{dsfont}

\usepackage{fancybox}
\usepackage{listings}
\usepackage{mathtools}
\usepackage{amsfonts}
\usepackage[bottom]{footmisc}
\usepackage{verbatim}

\usepackage{colortbl}

\usepackage{float}
\usepackage{makeidx}

\normalbaselines

\newtheorem{Theorem}{Theorem}[section]

\newtheorem{Proposition}[Theorem]{Proposition}

\newtheorem{Lemma}[Theorem]{Lemma}

\newtheorem{Corollary}[Theorem]{Corollary}

\newtheorem{Remark}[Theorem]{Remark}

\newcommand{\RR}{\mathbb R}

\newcommand{\C}{\mathbb C}

\newcommand{\nat}{{{\rm I} \kern -.15em {\rm N} }}

\newcommand{\matr}{\mathcal{R}}

\newcommand{\be}{\begin{equation}}
\newcommand{\ee}{\end{equation}}
\newcommand{\beq}{\begin{eqnarray}}
\newcommand{\eeq}{\end{eqnarray}}
\newcommand{\beqs}{\begin{eqnarray*}}
\newcommand{\eeqs}{\end{eqnarray*}}
\newcommand{\bt}{\begin{Theorem}}
\newcommand{\et}{\end{Theorem}}
\newcommand{\br}{\begin{Remark}}
\newcommand{\er}{\end{Remark}}
\newcommand{\bc}{\begin{Corollary}}
\newcommand{\ec}{\end{Corollary}}
\newcommand{\bl}{\begin{Lemma}}
\newcommand{\el}{\end{Lemma}}
\newcommand{\bd}{\begin{definition}}
\newcommand{\ed}{\end{definition}}

\renewcommand{\geq}{\geqslant}
\renewcommand{\leq}{\leqslant}

\graphicspath{{figures/}}

\title{Bifurcation analysis of a coupled system between a transport equation and  an ordinary differential equation with time delay}
\author{
Serge Nicaise\footnote{Univ.  Polytechnique  Hauts-de-France, LAMAV - FR CNRS 2956, F-59313 Valenciennes, France,  (\texttt{Serge.Nicaise@uphf.fr}).}
\and
Alessandro Paolucci\footnote{Dipartimento di Ingegneria e Scienze dell'Informazione e Matematica, Universit\`{a} di L'Aquila, Via Vetoio, Loc. Coppito, 67010 L'Aquila Italy (\texttt{alessandro.paolucci2@graduate.univaq.it}).}
\and
Cristina Pignotti\footnote{Dipartimento di Ingegneria e Scienze dell'Informazione e Matematica, Universit\`{a} di L'Aquila, Via Vetoio, Loc. Coppito, 67010 L'Aquila Italy (\texttt{pignotti@univaq.it}).}
}

\date{}

\begin{document}

\textwidth=160 mm

\textheight=225mm

\parindent=8mm

\frenchspacing

\maketitle

\begin{abstract}
In this paper we analyze a coupled system between a transport equation and an ordinary differential equation with time delay (which is a simplified version of a model for kidney blood flow control). Through a careful spectral analysis we characterize the region of stability, namely the set of parameters for which the system is exponentially stable. Also, we perform a bifurcation analysis and determine some properties of the stable steady state set and the limit cycle oscillation region. Some numerical examples illustrate the theoretical results.

\end{abstract}

\vspace{5 mm}

\def\qed{\hbox{\hskip 6pt\vrule width6pt
height7pt
depth1pt  \hskip1pt}\bigskip}

\section{Introduction}

Systems with   time delay     often appear in many biological, electrical engineering systems and mechanical
applications \cite{hadeler:79,Diekmannetall:95, Mazanti, Peralta, Andrii, PeraltaK, Ammari2, Mesquita}.  In many cases, in particular for distributed parameter systems, even arbitrarily
small delays in the feedback may destabilize the system, see e.g.
\cite{datko:88,datkolagnesepolis,nicaise:06a,ammarietal:08}. The stability issue of
systems with delay is, therefore, of theoretical and practical importance.
For a rigorous presentation of the basic theory of delay differential equations we refer e.g. to \cite{Halanay, Diekmannetall:95}.

In  \cite{ForVerseptetall:15,Liu16}
a model for kidney blood flow control is studied. The kidney's main function is to act as a filter that removes waste products and excess fluid from the body.
Kidney regulates the balance of water, minerals and blood pressure through filtration, absorption and secretion of water and solutes across the surface of renal tubules, the nephrons. In \cite{kidneyIMA} is proposed a model that can be recast in the following form:

\begin{equation}
\label{model_kidney}
\begin{array}{l}
\displaystyle{ \partial_t c(x,t) = F^0(d(t))\partial_x c(x,t)+F^1( c(x,t)), \quad (x,t)\in [0,l]\times [0,+\infty),}\\
\displaystyle{ d'(t)=F^2(d(t), a(t)), \quad t\in [0,+\infty),}\\
\displaystyle{ a'(t)=F^3(c(l,t-\tau), d(t), a(t)), \quad t\in [0,+\infty),}\\
\displaystyle{c(0,t)=c_0,  \quad t\in [0,+\infty),}
\end{array}
\end{equation}
with suitable initial conditions. Here, $c(x,t)$ is the concentration of chloride ions in appropriate tubule's portion at position $x$ and time $t,$ $d(t)$ and $a(t)$ are the diameter and smooth muscle tone (activation), respectively, of the afferent arteriole at time $t.$ In particular, the derivative of $a$ depends on the chloride ions concentration at the end of the tubule's portion where the process happens, $x=l,$ at a time $t-\tau,$ where the positive constant $\tau$ is a time delay. Moreover, $F^i,$ $i=0, \dots, 3,$ are suitable functions depending on the problem's parameters. In \cite{ForVerseptetall:15,Liu16} the authors investigate the interaction among feedback mechanisms and they perform a bifurcation analysis in order to show the stability of the model in a suitable range of parameters. Numerical simulations, confirming the bifurcation results, are also illustrated.

In order to address from a rigorous mathematical point of view such a kind of models, we analyze here a simplified problem with only two variables, the activation $a(t)$ and the chloride ions concentration $c(x,t).$ In particular, we
 are interested in studying the evolution of $c:[0,l]\times [0,+\infty)\rightarrow \mathbb{C}$ and $a:[0,+\infty) \rightarrow \mathbb{C}$, with $l>0$, subject to the following coupled system:
\begin{equation}
\label{model}
\begin{array}{l}
\displaystyle{ \partial_t c(x,t) = -f\partial_x c(x,t)+\beta a(t)-\delta c(x,t), \quad (x,t)\in [0,l]\times [0,+\infty),}\\
\displaystyle{ a'(t)=c(l,t-\tau)-\alpha a(t), \quad t\in [0,+\infty),}\\
\displaystyle{ c(0,t)=0, \quad t\in [0,+\infty),}\\
\displaystyle{ c(l,s)=c_l(s), \quad s\in [-\tau ,0],}\\
\displaystyle{ c(x,0)=c_0(x), \quad x\in [0,l],}\\
\displaystyle{ a(0)=a_0>0,}
\end{array}
\end{equation}
where $\alpha, f >0$, $ \beta,\delta\in\RR$ and $\tau>0$ is the time delay.

The well-posedness of this problem is relatively easy to check by using semigroup theory. So our main goals are stability issues. Namely we want to characterize the set of parameters for which this system is exponentially stable or not. For that purpose, we first characterize the spectrum of the operator associated with this system, even in the case $\tau=0$, by showing that the eigenvalues of the operator are roots of a holomorphic function. Then according to  \cite{ForVerseptetall:15,Liu16}, we perform a bifurcation analysis, namely
we introduce the stable steady state region $\matr_-$ and the limit cycle oscillations region $\matr_+$,
and analyze some of their analytical properties. In particular we show that these sets are non empty and open. We further show that their boundaries are neglectible. We finally illustrate our theoretical results by some numerical examples.

The paper is organized as follows: first, in Section 2, we show that the system \eqref{model} is well-posed by using semigroup theory in the case $\tau>0$.  We show in Section 3 how to extend our previous results to the case $\tau=0$. Then, in Section 4, using the one-dimensional character of the differential equation in \eqref{model}, we characterize the spectrum of the operator.  Section 5 is devoted to  the bifurcation analysis and in Section 6 we end up with an illustrative example.

\section{Well-posedness results}

To show that problem  \eqref{model} is well-posed when $\tau>0$, as in \cite{nicaise:06a} (see also \cite{Batkai:05}), for any $t\geq 0$ we introduce the new variable
\begin{equation}
\label{z}
z(\rho,t):=c(l,t-\rho \tau), \quad \rho \in [0,1].
\end{equation}
Hence, system \eqref{model} is equivalent to the following one:
\begin{equation}
\label{model_2}
\begin{array}{l}
\displaystyle{ \partial _t c(x,t) = -f\partial_x c(x,t)+\beta a(t)-\delta c(x,t), \quad (x,t)\in [0,l]\times [0,+\infty),}\\
\displaystyle{ a'(t)= z(1,t)-\alpha a(t), \quad t\in [0,+\infty) ,}\\
\displaystyle{ \tau \partial_t z(\rho,t)=-\partial_\rho z(\rho, t), \quad (\rho,t)\in [0,1]\times [0,+\infty),}\\
\displaystyle{ c(0,t)=0, \quad t\in[0,+\infty),}\\
\displaystyle{ c(x,0)=c_0(x),\quad x\in [0,l],}\\
\displaystyle{ a(0)=a_0,}\\
\displaystyle{ z(\rho ,0)=c(l,-\rho\tau), \quad \rho\in [0,1],}\\
\displaystyle{ z(0,t)=c(l,t), \quad t\in [0,+\infty).}
\end{array}
\end{equation}
Consider the Hilbert space
$$
H:=L^2\left(0,l\right) \times \mathbb{C}\times L^2 \left(0,1 \right),
$$
endowed with the inner product
$$
 \left\langle
\left (
\begin{array}{l}
u\\
v\\
w
\end{array}
\right ),
\left (
\begin{array}{l}
\tilde u\\
\tilde v\\
\tilde w
\end{array}
\right )
\right\rangle_H:= \int_0^l u(x)\cdot \overline{\tilde{u}(x)} dx +v\cdot \overline{\tilde{v}}+f\tau \int_0^1 w(\rho)\cdot \overline{\tilde{w}(\rho)} d\rho,
$$
and associated norm $||\cdot||_{H}$.
Setting $U=(c,a,z)^\top$ we can rewrite \eqref{model_2} as
\begin{equation}
\label{model_3}
\begin{array}{l}
\displaystyle{ U'=AU,}\\
\displaystyle{ U(0)=U_0,}
\end{array}
\end{equation}
where the operator $A$ is defined by
$$
A\left(
\begin{array}{l}
c \\
a\\
z
\end{array}\right) = \left(
\begin{array}{l}
-f\partial_x c +\beta a -\delta c\\
\hspace{0.5 cm} -\alpha a+z(1) \\
\hspace{0.7 cm} -\frac{1}{\tau} \partial_\rho z
\end{array}\right),
$$
with domain
$$
\mathcal{D}(A)=\left\{ (c,a,z)\in H^1 ( (0,l) ) \times \mathbb{C} \times H^1 ( (0,1) )\ : \ c(0)=0, \ z(0)=c(l) \right\}.
$$
We have the following well-posedness result.
\begin{Theorem}
For any $U_0\in H$, there exists a unique solution $U\in C([0,+\infty); H)$ to \eqref{model_3}. Moreover, if $U_0\in \mathcal{D}(A)$, then
$$
U\in C([0,+\infty); \mathcal{D}(A)) \cap C^1 ([0,+\infty); H).
$$
\end{Theorem}
\begin{proof}
Denoting by $I$ the identity operator, we show that there exists a constant $k>0$ such that the operator $A-kI$ is dissipative, namely
$$
\Re \langle AU,U \rangle_H \leq k||U||_{H}^2.
$$
We have that
$$
\begin{array}{l}
\displaystyle{  \langle AU,U \rangle_H = \int_0^l \left( -f \partial_x c(x)+\beta a-\delta c(x) \right) \cdot \overline{c(x)} dx+ (-\alpha a+z(1))\cdot \overline{a}}\\
\hspace{2 cm}
\displaystyle{ -f\int_0^1 \partial_\rho z(\rho)\cdot \overline{z(\rho)} d\rho}\\
\hspace{1.75 cm}
\displaystyle{ = -\frac{f}{2} |c(l)|^2 +\beta a\int_0^l \overline{c(x)}dx -\delta \int_0^l |c(x)|^2 dx -\alpha |a|^2 }\\
\hspace{2 cm}
\displaystyle{+z(1)\overline{a} -\frac{f}{2}|z(1)|^2+\frac{f}{2}|z(0)|^2.}
\end{array}
$$
Using Young inequality and Cauchy-Schwarz inequality yield
$$
\begin{array}{l}
\displaystyle{ \Re \langle AU,U \rangle_H \leq  -\frac{f}{2}|z(0)|^2+\frac{\beta}{2}|a|^2 +\left( \frac{\beta }{2}l-\delta\right) \int_0^l |c(x)|^2 dx }\\
\hspace{2 cm}
\displaystyle{ -\alpha |a|^2 +\frac{|a|^2}{2f}+\frac{f}{2}|z(1)|^2 -\frac{f}{2}|z(1)|^2+\frac{f}{2}|z(0)|^2}\\
\hspace{2 cm}
\displaystyle{ \leq k ||U||_{H}^2,}
\end{array}
$$
for some $k>0$. This gives us the dissipativity of the operator $A-kI$.

In order to use the Lumer-Phillips theorem, we need to prove the surjectivity of $\lambda I-A$ for some $\lambda >0$, namely we need to prove that for any $(h,m,n)^\top\in H$ there exists $(c,a,z)^\top\in \mathcal{D}(A)$ such that
$$
(\lambda I-A) \left( \begin{array}{l}
c\\ a\\ z
\end{array} \right) = \left( \begin{array}{l}
h\\m\\n
\end{array}\right).
$$
This is equivalent to solve the following system:
\begin{eqnarray}
\lambda c+f\partial_x c-\beta a +\delta c &=& h, \label{surj_1} \\
(\lambda+\alpha) a-z(1)&=&m, \label{surj_2} \\
\lambda z+\frac{1}{\tau}\partial_\rho z&=&n. \label{surj_3}
\end{eqnarray}
From \eqref{surj_3} we obtain immediately
\begin{equation}\label{sn:expressionz}
z(\rho)= e^{-\lambda \tau \rho} \left( c(l) +\int_0^\rho \tau n(\sigma) e^{\lambda \tau \sigma} d\sigma\right),
\end{equation}
and so
\begin{equation}
\label{cond_surj_1}
z(1)=e^{-\lambda \tau} \left( c(l)+\int_0^1 \tau n(\sigma) e^{\lambda\tau \sigma} d\sigma\right).
\end{equation}
Substituting \eqref{cond_surj_1} in \eqref{surj_2}, we obtain
\begin{equation}\label{sn:expressiona}
a=\frac{1}{\lambda +\alpha}\left( m+e^{-\lambda \tau} \left( c(l) +\int_0^1 \tau n(\sigma)e^{\lambda \tau \sigma} d\sigma \right) \right),
\end{equation}
which, substituted in \eqref{surj_1}, gives us a solution $c$:
\begin{equation}\label{sn:expressionc(x)}
\begin{array}{l}
\displaystyle{c(x)=\beta e^{-\lambda\tau} \frac{1-e^{-\frac{(\lambda+\delta)x}{f}}}{(\lambda+\delta)(\lambda+\alpha)} c(l)+\frac{e^{-\frac{(\lambda+\delta)x}{f}}}{f}\int_0^x h(y)e^{\frac{\lambda+\delta}{f}y} dy}\\ \\
\hspace{1 cm}
\displaystyle{ + \beta m \frac {1-e^{-\frac{(\lambda+\delta)x}{f}}} {(\lambda+\alpha)(\lambda+\delta)}+\beta \tau \frac{  1- e^{-\frac{(\lambda+\delta)x}{f}}}{(\lambda+\alpha)(\lambda+\delta)} e^{-\lambda\tau}  \int_0^1 n(\sigma)e^{\lambda \tau\sigma}d\sigma,}
\end{array}
\end{equation}
for some $\lambda$ large enough so that
$G(\lambda)$ defined by
\[
G(\lambda):=1-\beta e^{-\lambda\tau} \frac{1-e^{-\frac{\lambda+\delta}{f}l}}{(\lambda+\alpha)(\lambda+\delta)},
\]
is different from zero. Indeed in such a case, evaluating the previous expression at $l$, we will get
\begin{equation}\label{sn:expressionc(l)}
\begin{array}{l}
\displaystyle{c(l)= \frac{e^{-\frac{(\lambda+\delta)l}{f}}}{f G(\lambda)}\int_0^l h(y)e^{\frac{\lambda+\delta}{f}y} dy+\beta m\frac { 1-e^{-\frac{(\lambda+\delta)l}{f}}} {(\lambda+\alpha)(\lambda+\delta)G(\lambda)}}\\ \medskip
\hspace{2.5 cm}\displaystyle{
+\beta  \tau \frac{  1- e^{-\frac{(\lambda+\delta)l}{f}}}{(\lambda+\alpha)(\lambda+\delta)G(\lambda)} e^{-\lambda\tau}  \int_0^1 n(\sigma)e^{\lambda \tau\sigma}d\sigma.}
\end{array}
\end{equation}
Inserting this expression in \eqref{sn:expressionz}, \eqref{sn:expressiona}
and \eqref{sn:expressionc(x)}, we find $(x, a,z)^\top \in \mathcal{D}(A)$,   solution of \eqref{surj_1}-\eqref{surj_3}
and the surjectivity is proved.

Therefore, by Lumer-Phillips theorem, $A-k I$ generates a strongly continuous semigroup of contractions in $H$. Consequently  $A$ generates a strongly continuous semigroup   in $H$,  which concludes  the proof of the theorem.
\end{proof}

\section{The case $\tau=0$}

In this section, we show that the previous existence result can be extended to the case $\tau=0$. This requires
the following Hilbert setting
$$
H_0:=L^2\left(0,l\right) \times \mathbb{C},
$$
endowed with the inner product
$$
 \left\langle
\left (
\begin{array}{l}
u\\
v
\end{array}
\right ),
\left (
\begin{array}{l}
\tilde u\\
\tilde v
\end{array}
\right )
\right\rangle_{H_0}:= \int_0^l u(x)\cdot \overline{\tilde{u}(x)} dx +v\cdot \overline{\tilde{v}},
$$
and associated norm $||\cdot||_{H_0}$.
Setting $U=(c,a)^\top$ we can rewrite \eqref{model_2} with $\tau=0$ as
\begin{equation}
\label{model_30}
\begin{array}{l}
\displaystyle{ U'=A_0U,}\\
\displaystyle{ U(0)=U_0,}
\end{array}
\end{equation}
where the operator $A_0$ is defined by
$$
A_0\left(
\begin{array}{l}
c \\
a
\end{array}\right) = \left(
\begin{array}{l}
-f\partial_x c +\beta a -\delta c\\
\hspace{0.5 cm} c(l)-\alpha a
\end{array}\right),
$$
with domain
$$
\mathcal{D}(A_0)=\left\{ (c,a)\in H^1 (0,l ) \times \mathbb{C} \ : \ c(0)=0 \right\}.
$$

\begin{Theorem}
For any $U_0\in H_0$, there exists a unique solution $U\in C([0,+\infty); H_0)$ to \eqref{model_30}. Moreover, if $U_0\in \mathcal{D}(A)$, then
$$
U\in C([0,+\infty); \mathcal{D}(A_0)) \cap C^1 ([0,+\infty); H_0).
$$
\end{Theorem}
\begin{proof}
We will use, as for delayed model, the Lumer-Philips' theorem. First of all we show that there exists a constant $\gamma>0$ such that
\begin{equation}
\label{dissipativity}
\Re \langle A_0 U, U\rangle_{H_0} \leq \gamma ||U||_{H_0}^2.
\end{equation}
By definition of $\langle \cdot,\cdot\rangle_{H_0}$, we have that
$$
\begin{array}{l}
\displaystyle{
\langle A_0 U,U\rangle_{H_0} =-f\int_0^l \partial_x c(x) \cdot \overline{c(x)}dx+\beta a \int_0^l \overline{c(x)}dx-\delta \int_0^l |c(x)|^2 dx+c(l)\overline{a}-\alpha |a|^2.}
\end{array}
$$
Using Young and H\"older inequalities yields
$$
\begin{array}{l}
\displaystyle{\Re \langle A_0U,U\rangle_{H_0} \leq \left( \frac \beta 2 +\frac{1}{2f}-\alpha\right) |a|^2 + \left( \frac{\beta l}{2}-\delta \right) \int_0^l |c(x)|^2 dx.}
\end{array}
$$
Choosing
$$
\gamma=\max \left\{ \left| \frac \beta 2 +\frac{1}{2f}-\alpha \right| , \left| \frac{\beta l}{2}-\delta \right| \right\} ,
$$
we obtain \eqref{dissipativity}. In addition, we show that there exists a constant $\lambda\in\mathbb{C}$ such that $\lambda I-A_0$ is surjective, namely for any $(h,m)^\top \in H_0 $ there exists $(c,a)^\top \in \mathcal{D}(A_0)$ such that
\begin{equation}
\label{surjectivity1}
\begin{array}{l}
\displaystyle{ \lambda c +f\partial_x c-\beta a+\delta c =h,}\\
\displaystyle{ \lambda a -c(l)+\alpha a=m.}
\end{array}
\end{equation}
By the second equation of \eqref{surjectivity1} we have that for $\lambda\neq -\alpha$,
\begin{equation}
\label{sur1}
a=\frac{c(l)+m}{\lambda+\alpha},
\end{equation}
which, substituted in the first equation of \eqref{surjectivity1}, yields
\begin{equation}
\label{sur2}
c(x)=\frac 1 f e^{-\frac{\lambda+\delta}{f}x} \int_0^x h(y)e^{\frac{\lambda+\delta}{f}y} dy+\beta \frac{c(l)+m}{(\lambda+\alpha)(\lambda+\delta)}\left( 1-e^{-\frac{\lambda+\delta}{f}x}\right).
\end{equation}
If $x=l$, equation \eqref{sur2} reads as
$$
\left( 1-\beta \frac{1-e^{-\frac{\lambda+\delta}{f}l}}{(\lambda+\alpha)(\lambda+\delta)}\right) c(l)=\frac 1 f e^{-\frac{\lambda+\delta}{f}l} \int_0^l h(y) e^{\frac{\lambda+\delta}{f}y}dy+\beta \frac{m}{(\lambda+\alpha)(\lambda+\delta)}\left( 1-e^{-\frac{\lambda+\delta}{f}l}\right).
$$
Therefore, as before, we choose $\lambda$ big enough so that the function
$$
G_0(\lambda)= 1-\beta \frac{1-e^{-\frac{\lambda+\delta}{f}l}}{(\lambda+\alpha)(\lambda+\delta)}
$$
is different from $0$. Therefore, substituting $c(l)$ in \eqref{sur1} and \eqref{sur2}, it's possible to find a solution $(c,a)^\top$ such that \eqref{surjectivity1} holds. Hence, surjectivity is proved. Then, by Lumer-Philips' Theorem we obtain the thesis of the theorem.
\end{proof}

\section{Spectral Analysis}
In this section we want to find some characterization of the eigenvalues of $A$ with respect to the parameters appearing in \eqref{model}. First, we have the following lemma.
\begin{Lemma}
Consider $\beta=0$. Then, we have that $Sp(A)=\{-\alpha\}$, for any choice of the parameters $\alpha, \ l,\ f$ and $\tau$.
\end{Lemma}
\begin{proof}
First assume that $\tau>0$, then
if $\beta=0$, $\lambda\in\C$ is an eigenvalue of $A$ if and only if $\lambda$ satisfies
\begin{eqnarray}\label{spectralsystembeta=0}
\left\{
\begin{array}{lll}
\lambda c+f\partial_x c +\delta c &=& 0, \\
(\lambda+\alpha) a-z(1)&=&0, \\
\lambda z+\frac{1}{\tau}\partial_\rho z&=&0.
\end{array}
\right.
\end{eqnarray}
From the third equation we obtain
\begin{equation}\label{sn:exprz}
z(\rho)=c(l)e^{-\lambda\tau\rho}, \quad \text{for} \ \rho\in[0,1],
\end{equation}
which yields
$$
(\lambda+\alpha)a=c(l)e^{-\lambda\tau}.
$$
From the first equation we have that
$$
c(x)\equiv 0,\quad \text{for any}\ x\in[0,l],
$$
which gives us $c(l)=0$. Hence, $(\lambda+\alpha)a=0$. This means that the only possible eigenvalue of $A$ is $\lambda=-\alpha$ with eigenvector $(0,1,0)^\top$.

If  $\tau=0$ and $\beta=0$, then $\lambda\in\C$ is an eigenvalue of $A$ if and only if $\lambda$ satisfies
\begin{eqnarray*}
\left\{
\begin{array}{lll}
\lambda c+f\partial_x c +\delta c &=& 0, \\
(\lambda+\alpha) a-c(l)&=&0.
\end{array}
\right.
\end{eqnarray*}
This corresponds to
 \eqref{spectralsystembeta=0} where $z=c(l)$ (see \eqref{sn:exprz}), hence the proof is finishing as above.
\end{proof}
Now, consider the case $\beta\neq 0$. Then, we have the following lemma.
\begin{Lemma}
\label{Lemma_roots}
If $\beta\ne0$, $\lambda \in Sp(A)$ if and only if $\lambda\ne -\alpha$ and  $G(\lambda)=0$, where we recall that
\begin{equation}
\label{G_lambda}
G(\lambda):=1-\beta e^{-\lambda\tau} \frac{1-e^{-\frac{\lambda+\delta}{f}l}}{(\lambda+\alpha)(\lambda+\delta)},
\end{equation}
with here the convention that the ratio
\[
\frac{1-e^{-\frac{\lambda+\delta}{f}l}}{\lambda+\delta}=\frac{l}{f},
\]
if $\lambda= -\delta$.
\end{Lemma}
\begin{proof}
If $\tau>0$, the spectral equation takes here the form
\begin{eqnarray}\label{sn:spectrum}
\left\{
\begin{array}{lll}
\lambda c+f\partial_x c-\beta a +\delta c &=& 0, \\
(\lambda+\alpha) a-z(1)&=&0, \\
\lambda z+\frac{1}{\tau}\partial_\rho z&=&0.
\end{array}
\right.
\end{eqnarray}
Then, as before from the third equation, $z$ is given by \eqref{sn:exprz},
and substituted in the second equation yields for $\lambda\neq-\alpha$
$$
a=\frac{c(l)e^{-\lambda\tau}}{\lambda+\alpha}.
$$
Putting $a$ in the first equation, we obtain
$$
c(x)=\beta c(l)e^{-\lambda\tau}\frac{1-e^{-\frac{(\lambda+\delta)x}{f}}}{(\lambda+\delta)(\lambda+\alpha)},
\quad \forall \, x\in  [0,l].
$$
This identity implies that
$$
c(l)=\beta c(l)e^{-\lambda\tau}\frac{1-e^{-\frac{(\lambda+\delta)l }{f}}}{(\lambda+\delta)(\lambda+\alpha)},
$$
and therefore a non trivial constant $c(l)$ exists if and only if $\lambda$ satisfies $G(\lambda)=0$.

It remains to treat the case $\lambda=-\alpha$. In such a case, the second identity of \eqref{sn:spectrum} implies that $z(1)=0$, and by the third identity of \eqref{sn:spectrum}, we deduce that $z=0$. The first identity of \eqref{sn:spectrum} yields
$$
c(x)=\beta a \frac{1-e^{-\frac{(\lambda+\delta)x}{f}}}{\lambda+\delta}, \quad
\forall x\in [0,l].
$$
Since $c(l)=z(1)=0$, we find
\[
\beta a \frac{1-e^{-\frac{(\lambda+\delta)l}{f}}}{\lambda+\delta}=0,
\]
which implies that $a=0$ since $\beta, l$ and $\frac{1-e^{-\frac{(\lambda+\delta)l}{f}}}{\lambda+\delta}$
are different from zero. Hence $\lambda=-\alpha$ is not an eigenvalue of $A$.

If  $\tau=0$, then the spectral equation takes here the form
\begin{eqnarray*}
\left\{
\begin{array}{lll}
\lambda c+f\partial_x c-\beta a +\delta c &=& 0, \\
(\lambda+\alpha) a-c(l)&=&0,
\end{array}
\right.
\end{eqnarray*}
that again corresponds to
 \eqref{sn:spectrum} where $z=c(l)$ (see again \eqref{sn:exprz}). The remainder of the proof then remains unchanged.
 \end{proof}

\bc \label{eigenvalue_-delta}
If $\beta\ne 0$ and $\delta\ne \alpha$, then $\lambda=-\delta$ is an eigenvalue for $A$ if and only if $1-\frac{\beta le^{\delta \tau}}{f(\alpha-\delta)}=0.$
\ec
\begin{proof}
From the previous Lemma, under our assumptions, we know that
$G(-\delta)=1-\frac{\beta le^{\delta \tau}}{f(\alpha-\delta)}.$
The result immediately follows.
\end{proof}

\bc \label{eigenvalue_zero}
If   $\beta\ne 0$ and $\delta\ne 0$, then $\lambda=0$ is an eigenvalue of $A$ for all $\tau\geq 0$  if and only if
$\beta=
\frac{\alpha\delta}{1-e^{-\frac{\delta l}{f}}}.
$
On the contrary, if   $\beta\ne 0$ and $\delta=0$, then $\lambda=0$ is an eigenvalue of $A$ for all $\tau\geq 0$  if and only if
$\beta=
\frac{\alpha f}{l}.
$
\ec
\begin{proof}
As before, by \eqref{G_lambda},  $G(0)=0$ if and only if (with the convention
$
\frac{1-e^{-\frac{\delta}{f}l}}{\delta}=\frac{l}{f},
$
if $\delta=0$)
\[
1-\beta  \frac{1-e^{-\frac{\delta}{f}l}}{\alpha \delta}.
\]
This directly leads to the results.
\end{proof}

Let us now notice that if $\Re\lambda\ge0$, then $\lambda$ is contained in a finite ball. This is the aim of the following lemma.
\begin{Lemma}\label{bounded_eigenvalue}
Let $\lambda$ be an eigenvalue of $A$. If $\Re \lambda \geq 0$, then there exists a constant $C_{\beta,\delta}>0$ depending on $\beta$ and $\delta$ such that $|\lambda|\leq C_{\beta,\delta}.$
\end{Lemma}
\begin{proof}
If $\beta=0$, there is no eigenvalue $\lambda$ of $A$ such that $\Re \lambda \geq 0$, hence we can now assume that $\beta\ne0$.
Let $\lambda$ be an eigenvalue of $A$ with  $\Re \lambda \geq 0$. Then, by Lemma \ref{Lemma_roots}, as $G(\lambda)=0$
we have
\[
\lambda+\delta=
\beta e^{-\lambda\tau} \frac{1-e^{-\frac{\lambda+\delta}{f}l}}{\lambda+\alpha}.
\]
This implies that
\[
|\lambda+\delta|\leq
 \frac{2|\beta|}{|\lambda+\alpha|}\leq \frac{2|\beta|}{|\lambda|},
\]
because $|\lambda|\leq |\lambda+\alpha|.$
Hence by the triangle inequality, we get
$$
\begin{array}{l}
\displaystyle{|\lambda|\leq |\lambda+\delta|+|\delta|\leq \frac{2|\beta|}{|\lambda|}+|\delta|.}
\end{array}
$$
Hence, $0\leq |\lambda|\leq \frac{|\delta|+\sqrt{\delta^2+8|\beta|}}{2}=:C_{\beta,\delta}.$\end{proof}

Now, let us recall that  the spectral bound of the operator $A$ (see for instance \cite{MR1886588})
is defined by
$$
s(A)=\sup \ \{ \Re \lambda   \ :\ \lambda\in Sp(A)\},
$$
while
$$
s_0(A):=\inf \ \{ x> s(A) \ : \ \exists\, C_x>0 \ : \ ||R(\lambda,A)||\leq C_x \ \text{whenever} \ \Re \lambda > x\}.
$$
By Theorem 5.2.1 in \cite{MR1886588}, we know that
 $$\omega(T):=\inf \left\{ \omega \in \RR \ : \ \exists \, M_\omega >0 \ \text{such that} \ ||T(t)||\leq M_\omega e^{\omega t}, \ \forall t\geq0\right\} =s_0(A).$$ Moreover, we have the following result.
\begin{Theorem}
$s_0(A)=s(A).$
\end{Theorem}
\begin{proof}
We show that for all $x_0>s(A)$ there exists a constant $C_{x_0}>0$ such that
\begin{equation}
\label{sn:resolestimate}
||R(\lambda,A)||\leq C_{x_0}, \forall \,\lambda\in \C\quad\mbox{\rm such that}\quad \Re\lambda>x_0.
\end{equation}
For $(h,m,n)^\top\in H$, the vector
$$
 \left( \begin{array}{l}
c\\ a\\ z
\end{array} \right) = R(\lambda,A) \left( \begin{array}{l}
h \\m \\ n
\end{array} \right)
$$
is actually solution of \eqref{surj_1}-\eqref{surj_3}. Consequently,
$x, a$ and $z$ are respectively given by \eqref{sn:expressionc(x)}, \eqref{sn:expressiona}
and \eqref{sn:expressionz} with $c(l)$ given by
\eqref{sn:expressionc(l)}. So, in a first step we need to estimate from below $|G(\lambda)|$.

Hence fix $x_0>s(A)$, and let $\lambda=x+i\omega$, with arbitrary  $x> x_0$ and $\omega\in \mathbb{R}$, by \eqref{G_lambda} we notice that
\begin{eqnarray}
\nonumber
|G(\lambda)-1|&=&|\beta| e^{-x\tau}\frac{|1-e^{-\frac{(\lambda+\delta)l}{f}}|}{|\lambda+\alpha||\lambda+\delta|}
\\
&\leq &|\beta| e^{-x\tau}\frac{1+e^{-\frac{(x+\delta)l}{f}}}{((x+\alpha)^2+\omega^2)^\frac{1}{2}
((x+\delta)^2+\omega^2)^\frac{1}{2}}\,.\label{sn:1}
\end{eqnarray}
First, for $x$ large enough, we  notice that
\begin{eqnarray*}
|G(\lambda)-1| \leq  |\beta| e^{-x\tau}\frac{1+e^{-\frac{(x+\delta)l}{f}}}{(x+\alpha)
(x+\delta)},
\end{eqnarray*}
and since this right-hand side tends to zero as $x$ goes to infinity, there exists $x_1>0$ large enough such that
\begin{eqnarray*}
|G(\lambda)-1| \leq \frac{1}{2}, \quad \hbox{ for } x\geq x_1,
\end{eqnarray*}
which implies
\begin{eqnarray*}
|G(\lambda)| \geq \frac{1}{2}, \quad \hbox{ for } x\geq x_1.
\end{eqnarray*}

Now for $x\in [x_0, x_1]$, and $\omega\ne 0$, \eqref{sn:1} implies that
\[
|G(\lambda)-1|\leq |\beta| e^{-x_0\tau}\frac{1+e^{-\frac{(x_0+\delta)l}{f}}}{\omega^2}.
\]
As this right-hand side tends to zero as
$|\omega|\rightarrow +\infty$, there exists $\lambda_{x_0}>0$ such that
\begin{eqnarray*}
|G(\lambda)| \geq \frac{1}{2}, \quad \hbox{ for }|\omega|\geq \lambda_{x_0}.
\end{eqnarray*}
Finally introduce the compact set
\[
K=\{x+i\omega: x\in [x_0, x_1] , \ |\omega|\leq \lambda_{x_0}\},
\]
and the mapping
\[
K\to \mathbb{R}: \lambda\to |G(\lambda)|.
\]
Since this mapping is continuous in $K$ and is different from zero on $K$,   there exists $\alpha_{x_0}>0$ such that
\[
 |G(\lambda)|\geq \alpha_{x_0}>0, \hbox{ for } \lambda \in K.
 \]
 All together we have shown that there exists a positive constant $m_{x_0}$ such that
 \[
 |G(\lambda)|\geq m_{x_0}>0, \ \forall \lambda\in \C, \ \Re\lambda \geq x_0.
 \]

 With the help of this estimate in \eqref{sn:expressionc(l)}, and using Cauchy-Schwarz's inequality we find that
 \[
 |c(l)|\leq D_{x_0} ||(h,m,n)^\top||_H,
 \]
 for some positive constant $D_{x_0}$. Using this estimate in \eqref{sn:expressionc(x)}, \eqref{sn:expressiona}
and \eqref{sn:expressionz} and again  Cauchy-Schwarz's inequality, we deduce that
 \[
 ||(c,a,z)^\top||_H\leq C_{x_0} ||(h,m,n)^\top||_H,
 \]
 for some positive constant $C_{x_0}$, which is exactly \eqref{sn:resolestimate}.
\end{proof}

\section{Bifurcation Analysis}
Let us consider the following set:
$$
\mathcal{R}:=\left\{ (\alpha,\beta,\delta,l,f,\tau)\in \RR^6 \ :\ \alpha,\beta, l \in (0,+\infty), \ \delta,f\in\RR, \ \tau\in[0,+\infty) \right\}.
$$
For any $X\in\mathcal{R}$, let us write $A_X$ the corresponding generator of $T(t)$ with given coefficients, and $G_X(\lambda)$ as in \eqref{G_lambda}. Moreover, we define the following sets:
\begin{eqnarray}
\mathcal{R}_- &=& \left\{ X\in \mathcal{R}\ : \ \forall\, \lambda \in Sp(A_X), \ \Re\lambda <0\right\},\\
\mathcal{R}_0 &=& \left\{ X\in \mathcal{R}\ : \ Sp(A_X)\cap i\RR \neq \emptyset \right\},\\
\mathcal{R}_+ &=& \left\{ X\in \mathcal{R}\ : \ \exists \,\lambda \in Sp(A_X) \ \mbox{\rm  with} \ \Re\lambda >0\right\}.
\end{eqnarray}

We want to show that $\mathcal{R}_-$ is not empty. To this purpose, consider the following energy functional
\begin{equation}
\label{energy}
E(t)=\frac{1}{2}\int_0^l |c(x,t)|^2dx+\frac{1}{2}|a(t)|^2+\frac{\gamma}{2}\int_{t-\tau}^t e^{-(t-s-\tau)}|c(l,s)|^2 ds,
\end{equation}
where $\gamma$ is a positive coefficient that we will fix later. Then we have the following result.
\begin{Theorem}\label{thm:expdecay}
Let $(c,a)$ be the solution to \eqref{model}. Suppose that
\begin{equation} \label{assumptions_constants}
f(2\alpha-\beta )>1,\qquad \delta> \frac{\beta l}{2} >0,
\end{equation}
and assume that $\tau$ satisfies
\begin{equation}
\label{condition_tau}
e^\tau < f(2\alpha-\beta).
\end{equation}
Then, there exist $C,K>0$ such that
\begin{equation}
\label{exp_decay}
E(t)\leq Ce^{-Kt}, \qquad \forall t\geq 0.
\end{equation}
\end{Theorem}
\begin{proof}
Differentiating \eqref{energy} and using \eqref{model} yield
$$
\begin{array}{l}
\displaystyle{\frac{dE(t)}{dt}= -f|c(l,t)|^2+\beta a(t)\int_0^l c(x,t)dx-\delta\int_0^l|c(x,t)|^2dx+a(t)c(l,t-\tau)}\\
\hspace{2 cm}
\displaystyle{-\alpha|a(t)|^2+\frac{\gamma}{2}
e^\tau |c(l,t)|^2-\frac{\gamma}{2}|c(l,t-\tau)|^2-\frac{\gamma}{2}\int_{t-\tau}^t e^{-(t-s-\tau)}|c(l,s)|^2 ds.}
\end{array}
$$
By Young and H\"older's inequalities we obtain
$$
\begin{array}{l}
\displaystyle{\frac{dE}{dt}\leq \left( -\frac{f}{2}+\frac{\gamma}{2}e^\tau \right) |c(l,t)|^2+\left( \frac{\beta}{2}+\frac{1}{2\gamma}-\alpha \right) |a(t)|^2}\\
\hspace{2.35 cm}
\displaystyle{+\left( \frac{\beta l}{2}-\delta\right) \int_0^l |c(x,t)|^2 dx -\frac{\gamma}{2}\int_{t-\tau}^t e^{-(t-s-\tau)}|c(l,s)|^2 ds.}
\end{array}
$$
Now, we want to show that there exists a constant $K>0$ independent of $t$ such that for any $t\geq 0$
\begin{equation}
\label{aim}
\frac{dE}{dt}\leq - KE(t).
\end{equation}
To do so, we need that
\begin{equation}
\label{tau1}
-\frac{f}{2}+\frac{\gamma}{2}e^\tau\leq 0,
\end{equation}
\begin{equation}\label{tau2}
\frac{\beta}{2}+\frac{1}{2\gamma}-\alpha<0,
\end{equation}
\begin{equation}\label{other_condition}
\frac{\beta l}{2}-\delta <0.
\end{equation}
Condition \eqref{other_condition} is satisfied by assumption \eqref{assumptions_constants}. At the same time \eqref{tau1} and \eqref{tau2} give us
\begin{equation}\label{C1}
\frac{1}{2\alpha-\beta}<\gamma\leq fe^{-\tau}.
\end{equation}
In order to obtain \eqref{C1}, we need
$$
\frac{1}{2\alpha -\beta}< fe^{-\tau},
$$
which is true for any $\tau$ satisfying \eqref{condition_tau}. Hence, choosing $\gamma \in \left( \frac{1}{2\alpha-\beta}, fe^{-\tau}\right]$, inequality \eqref{aim} holds for
$$
K:=\min \left\{ \frac \gamma 2,\alpha-\frac{\beta}{2}-\frac{1}{2\gamma}, \delta -\frac{\beta l}{2}\right\}.
$$
Therefore, \eqref{exp_decay} immediately follows and this concludes the proof of the theorem.
\end{proof}
This theorem shows that the semigroup generated by $A_X$ is exponentially stable if
$X\in \mathcal{R}$ satisfies \eqref{assumptions_constants}-\eqref{condition_tau}. For such an element $X$,
$s_0(A_X)<0$, and  proves the following lemma.
\begin{Lemma}
$\mathcal{R}_-\neq \emptyset$.
\end{Lemma}
\begin{Lemma}
We have that
\begin{enumerate}
\item $\mathcal{R}_0\neq\emptyset$;
\item $\matr _+\neq\emptyset$.
\end{enumerate}
\end{Lemma}
\begin{proof}
In order to prove the first statement, we can use Corollary \ref{eigenvalue_zero} with $\delta=0$. So, we can say that if $X=\left( \frac{\beta l}{f}, \beta ,0,l,f,\tau\right)$, then $\lambda=0$ is an eigenvalue of $A_X$. Hence, $X\in\matr_0$.

Now, take again $\delta=0$, and $\alpha=f=1$. We notice that $\lambda=1$ is an eigenvalue if and only if
$$
\tau=\ln \frac{\beta ( 1-e^{-l})}{2}\,.
$$
This yields an element of $\matr _+$ provided
$$
\frac{\beta ( 1-e^{-l})}{2}>1.
$$
This proves the second statement.
\end{proof}
We have just shown that $\matr _+$ contains the points
$(1,\beta,0,l,1,\tau)$ provided $
\tau=\ln \frac{\beta ( 1-e^{-l})}{2}\,.
$   and $\beta$ is large enough. But this set is much larger as the next results show.
\begin{Lemma}\label{l:beta>beta_0}
For $\delta>0$ and $f>0$, let us set (see Corollary \ref{eigenvalue_zero})
\[
\beta_0(\alpha,\delta, l,f) :=\frac{\alpha\delta}{1-e^{-\frac{\delta l}{f}}}.
\]
Assume that
\begin{equation}\label{condG(0)}
e^{\frac{\delta}{f}l}-1-\frac{l}{f}\frac{\alpha\delta}{\alpha+\delta}\geq  0.
\end{equation}
Then
\be\label{eq:serge19/6:3}
\left\{ (\alpha,\beta,\delta,l,f,\tau)\in \matr : \ \delta>0,\  f>0, \ \beta>\beta_0(\alpha,\delta, l,f)\right\}\subset
\matr _+.
\ee
\end{Lemma}
\begin{proof}
Fix $\alpha>0$, $\delta>0$, $f>0$, and for shortness we skip the dependency in $\alpha,\delta, l$, and $f$. For all $\beta>\beta_0$,  we
 look for a positive real eigenvalue $x$ of $A$. Owing to Lemma
\ref{Lemma_roots} (since under our assumptions
$\beta_0$ is positive), $x\ge 0$  is an eigenvalue of $A$ if and only if
\be\label{eq:serge19/6:1}
e^{-x \tau}=\frac{h(x)}{\beta},
\ee
where
\[
h(x)=\frac{(x+\alpha)(x+\delta)}{1-e^{-\frac{x+\delta}{f}l}}, \ \forall x\geq 0.
\]
We first show that $h$ is an increasing function, by proving that its derivative is positive on $(0,\infty).$
Since for all $x\geq 0$, $1-e^{-\frac{x+\delta}{f}l}>0$, one readily checks that
$h'(x)>0$, for all  $x\geq 0$
if and only if
\be\label{eq:Gpos}
G(x):=e^{\frac{x+\delta}{f}l}-1-\frac{l}{f}\frac{(x+\alpha)(x+\delta)}{2x+\alpha+\delta}>0, \quad \forall x\geq 0.
\ee
Simple calculations show that
\[
G'(x)=\frac{l}{f}\left(
e^{\frac{x+\delta}{f}l}-1+\frac{2 (x+\alpha)(x+\delta)}{(2x+\alpha+\delta)^2}\right), \quad \forall x\geq 0.
\]
Since $e^{\frac{x+\delta}{f}l}-1$ is clearly positive, we deduce that
\[
G'(x)>0,  \ \forall x\geq 0.
\]
Therefore $G$ is an increasing function on $[0,\infty)$.
Since the assumption \eqref{condG(0)} means that $G(0)$ is non negative, we deduce that
\eqref{eq:Gpos} holds and
consequently $h$ is an increasing function.

Now as $h(0)=\beta_0$, we deduce that there exists a unique $x_\beta\in (0,\infty)$ such that
\[
h(x_\beta)=\beta.
\]
and
\[
\frac{h(x)}{\beta}<1, \quad \forall x\in (0, x_\beta).
\]
This means that for all $(0, x_\beta]$, there exists a unique $\tau(x)\geq 0$ such that
\[
e^{-x \tau(x)}=\frac{h(x)}{\beta},
\]
which is given by
\[
\tau(x):=-\frac{\ln\left(\frac{h(x)}{\beta}\right)}{x}.
\]
According to \eqref{eq:serge19/6:1} we get
\be\label{eq:serge19/6:2}
 \bigcup_{x\in  (0, x_\beta]}\{(\alpha,\beta,\delta,l,f,\tau(x))\}\subset
\matr _+, \quad \forall \beta>\beta_0.
\ee

But the function $\tau$  is clearly a continuous   function in $x\in (0, x_\beta]$ with
$\tau(x_\beta)=0$. Furthermore one easily checks that $\tau$ is   non increasing in $x$
and
\[
\lim_{x\to 0+}\tau(x)=+\infty.
\]
 Consequently the set
 \[
 \bigcup_{x\in  (0, x_\beta]} \{(\alpha,\beta,\delta,l,f,\tau(x))\}=
 \{(\alpha,\beta,\delta,l,f,\tau): \tau\geq 0\},
 \]
 and by \eqref{eq:serge19/6:2}, we deduce that
 \[
  \{(\alpha,\beta,\delta,l,f,\tau): \tau\geq 0\}\subset
\matr _+, \quad \forall \beta>\beta_0.
\]
This proves \eqref{eq:serge19/6:3}.
\end{proof}

\begin{Corollary}\label{c:beta>beta_0}
For $\delta>0$ and $f>0$, assume that
\begin{equation}\label{condG(0)negatiive}
e^{\frac{\delta}{f}l}-1-\frac{l}{f}\frac{\alpha\delta}{\alpha+\delta}< 0.
\end{equation}
Then with  $\beta_0(\alpha,\delta, l,f)$ defined above, there exists a positive real number  $x_0$ (that depends on $\alpha,\delta, l,$ and $f$) such that
\be\label{eq:serge30/6:1}
\left\{ (\alpha,\beta,\delta,l,f,\tau)\in \matr : \delta>0, f>0, \beta>\beta_0(\alpha,\delta, l,f),
\tau \in  \left[0, -\frac{\ln\left(\frac{\beta_0(\alpha,\delta, l,f)}{\beta}\right)}{x_0}\right]\right\}\subset
\matr _+.
\ee
\end{Corollary}
\begin{proof}
The proof is the same as the one of Lemma \ref{l:beta>beta_0}, the difference relies on the fact that now $G(0)$ being negative and since $G(x)$ tends to infinity as $x$ goes to infinity,
$h$ will be decreasing in an interval $(0, x_m)$ and increasing on $(x_m, \infty)$. Since $h$ blows up at infinity, there exists positive real number  $x_0$ such that
$h(x_0)=\beta_0$ and as before for all $\beta>\beta_0$
there exists $x_\beta\in (0,\infty)$ such that
\[
h(x_\beta)=\beta.
\]
and
\[
\frac{h(x)}{\beta}<1,\quad  \forall x\in (x_0, x_\beta).
\]
The  proof is finishing as before, the only difference is that the limit of $\tau(x)$ as $x$ goes to $x_0$
is no more $+\infty$ but here $-\frac{\ln\left(\frac{\beta_0}{\beta}\right)}{x_0}$.
\end{proof}

We now check that the sets $\matr _+$ and $\matr _-$ are open and   describe some properties of their boundary.
\begin{Lemma}\label{Rplus_open}
$\matr _+$ is an open subset of $\matr$.
\end{Lemma}
\begin{proof}
Fix $X_0\in\matr_+$. Then, there exists  $\lambda_0\in S_p(A_{X_0})$ such that $\Re\lambda_0>0$.
We define $d_X(\lambda):= |G_{X_0}(\lambda)-G_X(\lambda)|$ for any $\lambda $ such that
$\Re\lambda>0$ and for any $X\in \matr$. Since for $\epsilon>0$ fixed small enough $G_{X_0}(\lambda)\neq 0$ for any $\lambda\in \partial B(\lambda_0,\epsilon)$, then there exists $\overline{\delta}>0$ such that
$$
|G_{X_0}(\lambda)|\geq \overline{\delta},
$$
for any $\lambda\in \partial B(\lambda_0,\epsilon)$. Now, since $d_X(\lambda)$ is uniformly continuous on $\mathcal{B}:=\overline{B(X_0,r)}\times \partial B(\lambda_0,\epsilon)$ for any $r>0$, then we have that for any $\epsilon_0>0$, there exists $\delta=\delta(\epsilon_0)>0$ such that for any $(X,\lambda), (X_0,\lambda)\in \mathcal{B}$ with $||X-X_0||\leq \delta$,
$$
|d_X(\lambda)-d_{X_0}(\lambda)|=|d_X(\lambda)|\leq \epsilon_0.
$$
If we take $\epsilon_0< \overline{\delta}$, then we can apply Rouch\'e theorem and  obtain that  any $X\in B(X_0,\delta)$ belongs to $\matr_+$.
\end{proof}
\begin{Lemma}\label{R0_closed}
$\matr_0$ is a closed subset of $\matr$.
\end{Lemma}
\begin{proof}
Consider $X_n\in \matr_0$ such that
$$
\lim_{n\rightarrow +\infty} X_n=X.
$$
We show that $X\in \matr_0$. For any $n\in \nat$ there exist $\omega_n$ such that $i\omega_n \in Sp(A_{X_n})$.
By Lemma \ref{bounded_eigenvalue}, we know that  the set $\{\omega_n: n\in \nat\}$ is bounded. Hence, there exists a subsequence $\omega_{n_k}$ of $\omega_n$ such that $\omega_{n_k}\rightarrow \omega$. Now, as
$$
G_{X_n}(i\omega_n)\rightarrow G_X(i\omega), \quad n\rightarrow +\infty,
$$
and since $G_{X_n}(i\omega_n)=0$ for any $n\in\nat$, we obtain that $G_X(i\omega)=0$, which means that  $i\omega\in Sp(A_X)$. This concludes the proof of this lemma.
\end{proof}
\begin{Lemma}\label{Rminus_open}
$\matr_-$ is an open subset of $\matr$.
\end{Lemma}
\begin{proof}
We first notice that the complement $\matr_-^C$ of $\matr_-$
is
\[
\matr_-^C=\left\{ X\in \mathcal{R}\ : \ \exists\, \lambda \in Sp(A_X) \ : \ \Re\lambda \geq 0\right\}.
\]
Since  $\matr_-$ is   open if and only if $\matr_-^C$ is closed,
 it suffices to show that $\matr_-^C$ is closed.
 For that purpose, let  $X\in \matr$ and a sequence $X_n\in\matr_-^C$ such that $X_n\to X$, as $n\to +\infty$. Since $X_n\in\matr_-^C$,   there exists $\lambda_n\in Sp(A_{X_n})$ such that $\Re\lambda_n\geq 0$. By Lemma \ref{bounded_eigenvalue} and the fact that the sequence $X_n$ is bounded  there exists a constant $C>0$ (independent of $n$) such that
$$
|\lambda_n|\leq C.
$$
Hence, there exists a subsequence $\lambda_{n_k}$ of $\lambda_n$ such that $\lambda_{n_k}\to \lambda$, which gives us $\lambda\in Sp(A_X)$. Since
\[
\Re \lambda=\lim_{n\to\infty} \Re \lambda_{n_k},
\]
we directly deduce that $\Re \lambda\geq 0$, which proves that $X\in\matr_-^C$.
\end{proof}
\begin{Proposition}\label{null_measure}
$\matr_0$ is measurable and $meas(\matr_0)=0$.
\end{Proposition}
\begin{proof}
Let us introduce the  map
$$
\Psi:\matr \times \RR \to \C: (X, \omega)\to G_X(i\omega).
$$
Since $\Psi$ is clearly continuous,
 the set $\Psi^{-1}(\{0\})$ is measurable in $\matr\times\RR$.
 Now we notice that
 \[
 \matr_0=\{ X\in \matr \ :\ \exists \,\omega \in \RR \ : \ \Psi(X,\omega))=0\},
 \]
 or equivalenetly $ \matr_0$ is the projection of $\Psi^{-1}(\{0\})$ on $\matr$.
By the measurable  projection theorem we deduce that $\matr_0$ is measurable. So, Fubini's theorem yields
$$
meas(\matr_0)=\int_{\matr} \mathds{1}_{\matr_0} dX =\int_{(\alpha,\beta,\delta,l,f)}\left( \int_0^{+\infty} \mathds{1}_{\matr_0} d\tau\right) d\alpha d\beta d\delta dl df.
$$
Now, we claim that, $ \forall\, \alpha, f, l>0, \ \beta, \delta \in\RR,$ the set
$$
\matr_{0, \alpha, \beta,\delta, f,l}:=\{ \tau\geq 0 : \ \exists\,\omega\in\RR\ : \ G_X(i\omega)=0\}
$$
is countable. Indeed, let us fix $\alpha,\beta,\delta,f,l$ and let $\tau$ vary. Moreover, assume that there exists $\omega\in\RR$ such that $i\omega\in Sp(A_X)$. Hence,
\begin{equation}
\label{sn:eqspectrumiR}
\beta e^{-i\omega\tau} \frac{1-e^{-\frac{i\omega+\delta}{f}l}}{(i\omega+\alpha)(i\omega+\delta)}=1.
\end{equation}
Then, by taking the absolute value, we have
$$
\beta^2 \left( 1+e^{ -\frac{2\delta l}{f}} -2 e^{ -\frac{\delta l}{f}} \cos \left( \frac{\omega l}{f}\right) \right) =\omega^4+\omega^2 \left( (\alpha+\delta)^2-2\alpha\delta \right) +\delta^2\alpha^2,
$$
which is satisfied only for a finite number (call it $I\geq 0$) of $\omega_j\in\RR$, $j=0,\dots, I$
(note that $I$ and $\omega_j$ depend on  $\alpha,\beta,\delta,l,f$). Then coming back to \eqref{sn:eqspectrumiR}, we find
$$
e^{-i\omega_j\tau}=\frac{1}{\beta} \frac{(i\omega_j+\alpha)(i\omega_j+\delta)}{1-e^{-\frac{i\omega_j+\delta}{f}l}}=e^{i\theta_j},
$$
where $\theta_j\in\RR$ depend on $\alpha,\beta,\delta,l,f,\omega_j$, for any $j=0,\dots,I$. Hence, $\omega_j \tau =-\theta_j+2k\pi$, for any $k\in\mathbb{Z}$. Therefore, the set
$
\matr_{0, \alpha, \beta,\delta, f,l}
$
is indeed countable. Finally
as
\[
meas(\matr_0)=\int_{(\alpha,\beta,\delta,l,f)}\left( \int_0^{+\infty} \mathds{1}_{\matr_0} d\tau\right) d\alpha d\beta d\delta dl df
=\int_{(\alpha,\beta,\delta,l,f)}\left( \int_0^{+\infty} \mathds{1}_{\matr_{0, \alpha, \beta,\delta, f,l}} d\tau\right) d\alpha d\beta d\delta dl df,
\]
 we can conclude that $meas(\matr_0)=0.$
\end{proof}
\begin{Lemma}\label{l:bdyR+R-}
The following inclusions hold:
\begin{enumerate}
\item $\partial \matr_-\subset \{ X\in\matr \ : \ \forall \lambda \in Sp(A_X) \ : \Re\lambda \leq 0\} \cap \matr_0.$
\item  $\partial \matr_+ \subset \{ X\in\matr \ : \ \forall \lambda \in Sp(A_X) \ : \Re\lambda \leq 0\} \cap \matr_0.$
\end{enumerate}
\end{Lemma}
\begin{proof}
First of all, since by Lemma \ref{Rminus_open}, $\matr_-$ is open, we have that $\partial \matr_-=\overline{\matr_-}\setminus \matr_-$. Therefore, we prove that
\begin{equation}\label{I1}
\overline{\matr_-}\subset \{X\in\matr \ : \ \forall\,\lambda\in Sp(A_X),\ \Re\lambda\leq 0\}.
\end{equation}
Let $X\in \overline{\matr_-}$. Then, there exists a sequence $X_n\in \matr_-$ such that $X_n\to X$ as $n\to +\infty$. Since $X_n\in \matr_-$, then for any $n$ there exists $\lambda_n \in S_p(A_{X_n})$ such that $\Re\lambda_n<0$. Let $\lambda\in Sp(A_X)$ and suppose by contradiction that $\Re\lambda>0$.  Then, as in Lemma \ref{Rplus_open}, we can apply Rouch\'e theorem and we get that there exists $\epsilon>0$ such that for any $n\in\nat$ there exists $\lambda_n\in Sp(A_{X_n})$ with $\lambda_n\in B(\lambda,\epsilon)\subset \{ \lambda\in\C \ : \ \Re\lambda>0\}$, which is in contradiction with the fact that $X_n\in\matr_-$. Therefore, we get \eqref{I1}. Hence, $\overline{\matr_-}\subset \{ X\in\matr \ : \ \forall \,\lambda\in Sp(A_X),\ \Re\lambda\leq 0\}$ and
$$
\begin{array}{l}
\displaystyle{ \partial\matr_- \subset \{ X\in\matr \ : \ \forall \,\lambda\in Sp(A_X),\ \Re\lambda\leq 0\} \setminus \{ X\in\matr\ : \ \forall\, \lambda\in Sp(A_X), \ \Re\lambda<0\} }\\
\hspace{0.9 cm}
\displaystyle{ =\{ X\in\matr \ :\ \forall\,\lambda\in Sp(A_X), \ \Re\lambda\leq 0\}\cap \matr_0.}
\end{array}
$$
The second statement is a direct consequence of Proposition \ref{null_measure}.
Indeed
$\matr_+$ is clearly a subset of $\matr_-^C$
and consequently
$$
\overline{\matr_+}\subset \matr_-^C.
$$
Using the fact that $\partial\matr_+=\overline{\matr_+}\setminus \matr_+$ (where we used Lemma \ref{Rplus_open}), we get statement $2$.
\end{proof}
Thanks to this lemma together with Proposition \ref{null_measure}, we have the following corollary.
\begin{Corollary}
$\partial\matr_-$ and $\partial\matr_+$ are negligible sets.
\end{Corollary}

\section{Numerical illustration}

To illustrate our theoretical results, we have fixed $\alpha=\delta=l=f=1$ and let $\beta$ and $\tau$ vary.
First notice that by Theorem \ref{thm:expdecay},
the conditions \eqref{assumptions_constants} and \eqref{condition_tau} reduces
to $0<\beta<1$ and
$e^\tau<2-\beta$, so that
the set
\be\label{l:regionThm4.1}
\{(1,\beta,1,1,1,\tau): \beta<1 \hbox { and } 0\leq \tau<\log (2-\beta)\}\subset \matr_-.
\ee
This yields a stable steady state region, namely a region of pairs $(\beta, \tau)$ in $\matr_-$.

Second, by Lemma \ref{l:beta>beta_0}, the region
\[
\left\{(1,\beta,1,1,1,\tau): \beta>\frac{1}{1-e^{-1}} \hbox { and }   \tau\in [0,\infty)\right\}\subset \matr_+,
\]
which furnishes a  limit cycle oscillation  region.

 To determine the steady state or limit cycle oscillation  of the other  regions
of the half-plane
\[
\{(1,\beta,1,1,1,\tau): \beta\in \mathbb{R} \hbox { and }   \tau\geq 0\},
\]
we first characterized numerically the region  $\matr_0$. To do so, we look for
 $\omega\in\RR$ such that $i\omega\in Sp(A_X)$, hence solution of
 \eqref{sn:eqspectrumiR} (with $\alpha=\delta=l=f=1$), which is equivalent to
 \[
\beta=e^{i\omega\tau} \frac{(i\omega+1)^2}{1-e^{-(i\omega+1)}}.
\]
 Taking the real part and the imaginary part of the right-hand side, we find
\be\label{eq:serge19/6:4}
\beta=\frac{(1-e^{-1} \cos \omega)\left((1-\omega^2) \cos(\omega \tau)-2\omega \sin(\omega \tau)\right) +e^{-1} \sin \omega \left((1-\omega^2) \sin(\omega \tau)+2\omega \cos(\omega \tau)\right)}{1-2 e^{-1} \cos \omega+e^{-2}}
\ee
and
\be\label{eq:serge19/6:5}
0=(1-e^{-1} \cos \omega)\left((1-\omega^2) \sin(\omega \tau)+2\omega \cos(\omega \tau)\right) -e^{-1} \sin \omega \left((1-\omega^2) \cos(\omega \tau)-2\omega \sin(\omega \tau)\right).
\ee

We notice that the right-hand side of \eqref{eq:serge19/6:5} does not depend on $\beta$,
then the idea is to solve \eqref{eq:serge19/6:5} for a finite numbers of $\tau\in [0,10]$, namely by taking
\[
\tau_p=(p-1)\frac{10}{499},
\]
with $1\leq p \leq 500$, we look for the   solutions $\omega_{p,j}\in\RR$, $j=0,\dots, I_p$
of \eqref{eq:serge19/6:5} with the help of the Matlab function \emph{roots}. Then, coming back to \eqref{eq:serge19/6:4}, we find some  $\beta_{p,j}$ and hence some
pairs $(\tau_p, \beta_{p,j})$ in $\matr_0$. These pairs are represented by dotted points in  Figure \ref{fig:illustration}.


\begin{figure}[tb]
\centering
\includegraphics[width=1.0\columnwidth]{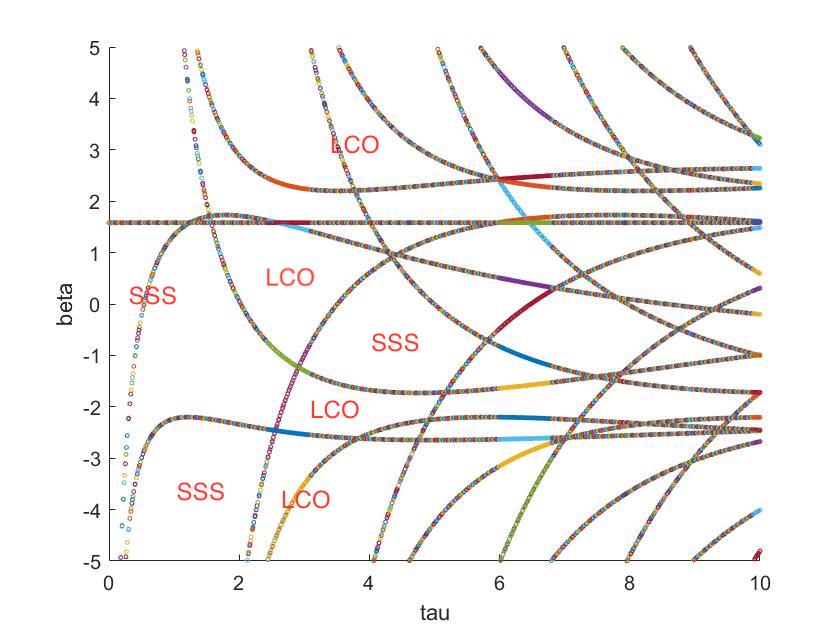}
\caption{Illustration of $\matr_0$, $\matr_-$ and $\matr_+$ with  $\alpha=\delta=l=f=1$. \label{fig:illustration}}
\end{figure}

Note that $\omega=0$ is a solution of \eqref{eq:serge19/6:5} for all $\tau\geq 0$ and therefore we find
$\beta_0= \frac{1}{1-e^{-1}}$.

Let us now make some comments:
\\
1. The curves corresponding to $\matr_0$ in the region  $\beta>\beta_0$ are immersed in $\matr_+$, but
this is not in contradiction with Lemma \ref{l:bdyR+R-}.
\\
2.
By the inclusion \eqref{l:regionThm4.1} and the fact that $\matr_-$ is open, the region around $(0,0)$ is a stable steady state region (in short SSS). On the contrary,  according to Lemma \ref{l:beta>beta_0}
the region $\beta>\beta_0=\frac{1}{1-e^{-1}}$  is a  limit cycle oscillation  region (in short LCO). To determine the nature of the neighboring region, for one fixed  point $(\tau,\beta)$ in such a region, we  use
the Matlab routine \emph{vpasolve} starting from  points $z_0$ randomly varying  in the domain $(\Re z_0, \Im z_0)\in [0,10]\times [-10,10]$ (see  Lemma \ref{bounded_eigenvalue}) that  returns a solution closed to the initial guess $z_0$. By letting $z_0$ vary, if we are not able to find a solution with positive real part, we deduce that the couple $(\tau,\beta)$ is such that $(1,\beta,1,1,1,\tau)$ belongs to $\matr_-$, otherwise it will be in $\matr_+$
and so the full region as well. By this algorithm, we have found that
the regions corresponding to the pairs $(\tau,\beta)=(1,1),\ (1,-3),\ (3,1),\ (4,-1),(3.9,1.1)$ are in $\matr_-$, while
the regions corresponding to the pairs $(\tau,\beta)=(1,3), \ (3,-3),\ (4,-2),(4,-4)$ are in $\matr_+$. We refer to Figure \ref{fig:illustration} for an illustration.



\section*{Acknowledgement}
We want to thank GNAMPA (INdAM) for partial financial support.

\end{document}